\documentclass{amsart}
\usepackage[utf8]{inputenc}
\usepackage{graphicx}
\usepackage{array}
\usepackage{tabularx}
\usepackage{amssymb}
\usepackage{tikz}
\usepackage{comment}
\usepackage{tikz-cd}
\usetikzlibrary{graphs}
\usepackage{comment}
\usepackage{blindtext}
\usepackage{hyperref}\hypersetup{colorlinks,allcolors=black}
\usepackage{varwidth}
\usepackage{import}
\usepackage{float}
\usepackage{amsmath}
\usepackage{amsthm}
\usepackage{calc}
 
\usepackage{caption}
\usepackage{subcaption}
\usepackage{wrapfig}

\usepackage{mathrsfs}

\newcommand{\handle}[1]{$#1$\text{-handle}}
\newcommand{\handles}[1]{$#1$\text{-handles}}
\newcommand{\calA}[1]{#1^{\mathcal{A}}}

\newcommand{\calS}[1]{#1^{\mathcal{S}}}

\newcommand{\halff}{\text{$\lfloor \frac{n}{2} \rfloor$}}

\makeatletter
\newtheorem*{rep@lemma}{\rep@title}
\newcommand{\newreplemma}[2]{%
\newenvironment{rep#1}[1]{%
 \def\rep@title{#2 \ref{##1}}%
 \begin{rep@lemma}}%
 {\end{rep@lemma}}}
\makeatother

\makeatletter
\newtheorem*{rep@definition}{\rep@title}
\newcommand{\newrepdefinition}[2]{%
\newenvironment{rep#1}[1]{%
 \def\rep@title{#2 \ref{##1}}%
 \begin{rep@definition}}%
 {\end{rep@definition}}}
\makeatother

\makeatletter
\newtheorem*{rep@corollary}{\rep@title}
\newcommand{\newrepcorollary}[2]{%
\newenvironment{rep#1}[1]{%
 \def\rep@title{#2 \ref{##1}}%
 \begin{rep@corollary}}%
 {\end{rep@corollary}}}
\makeatother

\makeatletter
\newtheorem*{rep@algorithm}{\rep@title}
\newcommand{\newrepalgorithm}[2]{%
\newenvironment{rep#1}[1]{%
 \def\rep@title{#2 \ref{##1}}%
 \begin{rep@algorithm}}%
 {\end{rep@algorithm}}}
\makeatother

\makeatletter
\newtheorem*{rep@theorem}{\rep@title}
\newcommand{\newreptheorem}[2]{%
\newenvironment{rep#1}[1]{%
 \def\rep@title{#2 \ref{##1}}%
 \begin{rep@theorem}}%
 {\end{rep@theorem}}}
\makeatother

\newtheorem{theorem}{Theorem}[section]

\newtheorem{lemma}[theorem]{Lemma}
\newtheorem{corollary}[theorem]{Corollary}
\newtheorem{proposition}[theorem]{Proposition}

\theoremstyle{definition}
\newtheorem{definition}[theorem]{Definition}
\newtheorem{remark}[theorem]{Remark}

\newtheorem{example}[theorem]{Example}

\newreptheorem{theorem}{Theorem}
\newreplemma{lemma}{Lemma}
\newrepcorollary{corollary}{Corollary}
\newrepcorollary{algorithm}{Algorithm}

\definecolor{amaranth}{rgb}{0.9, 0.17, 0.31} 
\definecolor{carrotorange}{rgb}{0.93, 0.57, 0.13} 
\definecolor{citrine}{rgb}{0.89, 0.82, 0.04} 
\definecolor{dartmouthgreen}{rgb}{0.05, 0.5, 0.06} 
\definecolor{ballblue}{rgb}{0.13, 0.67, 0.8} 
\definecolor{ceruleanblue}{rgb}{0.16, 0.32, 0.75} 
\definecolor{amethyst}{rgb}{0.6, 0.4, 0.8} 
\definecolor{amber}{rgb}{1.0, 0.75, 0.0} 
\definecolor{burlywood}{rgb}{0.87, 0.72, 0.53} 

\title{Handle decompositions and stabilizations of open books}
\author{Chun-Sheng Hsueh}
\address{Humboldt-Universit\"at zu Berlin, Rudower Chaussee 25, 12489 Berlin, Germany.}
\email{chun-sheng.hsueh@hu-berlin.de}
\begin{document}

\begin{abstract}
We build handle decompositions of $n$-manifolds that encode given open book decompositions and describe handle slides that reveal new open book decompositions on the same underlying manifold, for $n\geq3$. This recovers known stabilization operations for open books. As an application, we show that any open book with trivial monodromy can be stabilized to an open book whose page is a boundary connected sum of trivial disk bundles over spheres.
\end{abstract}


\makeatletter
\@namedef{subjclassname@2020}{%
  \textup{2020} Mathematics Subject Classification}
\makeatother


\maketitle

\section{Introduction}
An \textit{abstract open book decomposition} of a closed smooth $n$-manifold $X$ is a pair $(M,\varphi)$ consisting of a compact $(n-1)$-manifold $M$ with non-empty boundary and a diffeomorphism $\varphi\colon M\rightarrow M$ that fixes a neighborhood of the boundary pointwise such that $\operatorname{Ob}(M,\varphi)=M\times [0,1]/\sim$ is diffeomorphic to $X$, where $$(x,1)\sim(\varphi(x),0)\text{ for every }x\in M\text{, and}$$
$$(y,t)\sim(y,t')\text{ for every } t,t'\in [0,1]\text{, and }y\in\partial M.$$
The manifold $M$ is called the \textit{page}, its boundary $\partial M$ is referred to as the \textit{binding}, and $\varphi$ is known as the \textit{monodromy}. Two open book decompositions $(M_1,\varphi_1),(M_2,\varphi_2)$ are considered \textit{equivalent} if there exists a diffeomorphism $h\colon M_1\rightarrow M_2$ such that $h\circ \varphi_2=\varphi_1\circ h$. Every closed orientable $3$-manifold admits an open book decomposition~\cite{3ob}, while the existence of open book decompositions in higher dimensions is established~\cite{Lawson, quinn, Winkelnkemper} except in dimension four~\cite{kastenholz2025simplicialvolumeopenbooks}.


\subsection{Main results}
Theorem~\ref{thm technical} serves as the key technical result from which Theorems~\ref{thm stabilization of open books} and~\ref{thm odd stabilization} follow, bridging lines of work that had appeared unrelated. The former recovers open books discovered by Quinn~\cite[Section~4.4]{quinn}, and the latter corresponds to a standard stabilization in contact topology (see Remark~\ref{remark contact}).

\begin{theorem}\label{thm stabilization of open books}
    Let $(M,\varphi)$ be an open book decomposition on $X$ of dimension $n\geq 3$. Then for each integer $k\in[2,n-1]$, there exists a monodromy map $\varphi_k$ on $M \mathop{\natural} (S^{k-1}\times D^{n-k})\mathop{\natural} (S^{n-k}\times D^{k-1})$ such that $\varphi_k$ restricts to $\varphi$ on $M$, and
    $$(M\mathop{\natural} (S^{k-1}\times D^{n-k})\mathop{\natural} (S^{n-k}\times D^{k-1}),\varphi_k)$$
    is an open book decomposition on $X$ not equivalent to $(M,\varphi)$.
\end{theorem}

\begin{theorem}\label{thm odd stabilization}
    Let $(M,\varphi)$ be an open book decomposition on $X$, where the page $M$ has dimension $2\ell$ ($\ell\geq1$). Then there exists a monodromy map $\varphi'$ on $M\mathop{\natural}({S}^{\ell}\times D^{\ell})$ such that $\varphi'$ restricts to $\varphi$ on $M$, and $$(M\mathop{\natural}({S}^{\ell}\times D^{\ell}),\varphi')$$
    is an open book decomposition on $X$ not equivalent to $(M,\varphi)$.
\end{theorem}

We call the open book $(M\mathop{\natural} (S^{k-1}\times D^{n-k})\mathop{\natural} (S^{n-k}\times D^{k-1}),\varphi_k)$ the \textit{$k$-stabilization} of $(M,\varphi)$, and $(M\mathop{\natural}({S}^{\ell}\times D^{\ell}),\varphi')$ the \textit{middle-dimensional stabilization} of an odd-dimensional open book $(M,\varphi)$. Here we briefly discuss their relation to previously known stabilization. In a different vein, \cite{Generalized_plumbings} investigates methods for constructing new open books from existing ones, where the diffeomorphism type of the underlying manifolds is not preserved.

In dimension $n=3$, Harer conjectured in~\cite{Harer} that all open books on $S^3$ can be obtained by applying Hopf plumbings and deplumbings to the standard open book $(D^2, \operatorname{id})$ on $S^3$, a statement later proven by Giroux and Goodman~\cite{Giroux_Goodman}. We recall the definition of Hopf plumings in terms of handle decompositions in Example~\ref{ex dehn twist}. In Remark~\ref{remark on relation with other stabilizations}, we show that the middle-dimensional stabilization defined in Theorem~\ref{thm odd stabilization} is equivalent to a single Hopf plumbing, while the $2$-stabilization mentioned in Theorem~\ref{thm stabilization of open books} corresponds to performing Hopf plumbing twice. In general, one can associate a $2$-plane field to an open book on a $3$-manifold $M$ by taking the field of tangent planes to the pages away from the binding and extending it canonically near the binding, which can be viewed as a section of the unit cotangent bundle $ST^*M$. Two open book decompositions of a given $3$-manifold admit isotopic stabilization if and only if the open books define $2$-plane fields in $H_2(ST^*M)$ of the same homology class~\cite{Giroux_Goodman}. In higher dimensions, it remains unclear which stabilizations are necessary to relate open book decompositions on the same manifold.

When $n=4$, the $2$-stabilization provided by Theorem~\ref{thm stabilization of open books} coincides with the stabilization used in~\cite{trisecting_openbook}, where it is shown that this stabilization corresponds to the stabilization of a trisection diagram. The $2$- and $3$-stabilizations defined in Theorem~\ref{thm stabilization of open books} modify the binding by adding a copy of $S^2$. They differ from the stabilization procedure introduced in~\cite{CPV}, which arises in the study of $4$-dimensional open books associated with maximally non-integrable plane fields, known as \textit{Engel structures}. Their stabilization preserves the property of the binding being a union of tori -- a natural condition for Engel manifolds. 

For $p+q+1\geq5$, Quinn briefly pointed out in~\cite[Section~4.4]{quinn} the existence of open books on $S^{p+q+1}$ with pages $(S^p\times D^q)\mathop{\natural} (S^q\times D^p)$, without discussing its monodromy. Interestingly, Budney--Gabai's \emph{barbell diffeomorphisms}~\cite{budney_gabai} is the monodromy map on $S^2\times D^2\mathop{\natural} S^2\times D^2$ needed for constructing an open book on $S^5$. For each $n\geq 3$, we explicitly describe $(n-2)$ different open books on $S^n$ with such pages, which are also known as \emph{barbell manifolds}~\cite{budney_gabai_2}.
\begin{theorem}\label{thm book equivalent monodromies}
    For each integer $k\in[2,n-1]$ and $n\geq3$, there exists a monodromy map $\tau_{k}$ on $(S^{k-1}\times D^{n-k})\mathop{\natural} (S^{n-k}\times D^{k-1})$ such that $$((S^{k-1}\times D^{n-k})\mathop{\natural} (S^{n-k}\times D^{k-1}),\tau_{k})$$ is an open book decomposition on $S^n$. Moreover, for each integer $k\in[2,\halff]$ and $n\geq4$, the pair of monodromy maps $$\tau_k,\tau_{n-k+1}\colon (S^{k-1}\times D^{n-k})\mathop{\natural} (S^{n-k}\times D^{k-1})\rightarrow (S^{k-1}\times D^{n-k})\mathop{\natural} (S^{n-k}\times D^{k-1})$$ are non-isotopic (relative the boundary).
\end{theorem}

This provides examples of manifolds admitting open book decompositions with the same page but non-isotopic monodromy maps. Examples of $3$-manifolds that admit open book decompositions with an infinite family of pairwise non-isotopic monodromy maps on a fixed page can be found in~\cite{Stabilization_height,Stabilization_height_2,3mfd_nonisotopic_monodromies}.


\subsection{Application}
We show that every open book of dimension $n\geq4$ with trivial monodromy admits a handle decomposition that, after suitable handle slides, yields a handle decomposition of an open book whose page is a boundary connected sum of trivial disk bundles over spheres. This result generalizes a previously known phenomenon in dimension 4~\cite{hsueh2023kirby}.
    
\begin{theorem}\label{thm generalization}
    Let $(M,\operatorname{id})$ be an open book decomposition on a dimension $n\geq 4$ manifold $X$ with trivial monodromy. Let $\mu=(\mu_1,\mu_2,\dots,\mu_{n-2})\in\mathbb{N}_0^{n-2}$ be a tuple of non-negative integers such that $M$ admits a handle decomposition with one \handle{0}, $\mu_1$ \handles{1} and $\mu_i$ \handles{(n-i)}, $i=2,\dots,n-2$. Then there exists a monodromy map $$\sigma_\mu\colon\mathop{\natural}_{i=1}^{n-2}\mathop{\natural}_{\mu_i}(S^{i}\times D^{n-1-i})\rightarrow\mathop{\natural}_{i=1}^{n-2}\mathop{\natural}_{\mu_i}(S^{i}\times D^{n-1-i})$$ such that 
    $$(\mathop{\natural}_{i=1}^{n-2}\mathop{\natural}_{\mu_i}(S^{i}\times D^{n-1-i}),\sigma_\mu)$$
    is an open book decomposition on $X$.
\end{theorem}

\begin{corollary}\label{cor common page}
    Let $(M,\operatorname{id}_M)$ and $(N,\operatorname{id}_N)$ be open book decompositions on $n$-dimensional manifolds $X$ and $Y$, respectively. If
    \begin{equation*}
        \chi(M)-\chi(N)\begin{cases}
        =0 & \text{when $n-1$ is odd, or}\\
        \equiv 0\mod2 & \text{when $n-1$ is even,}
        \end{cases}
    \end{equation*}
    then there exists a tuple of non-negative integers $\nu=(\nu_1,\nu_2,\dots,\nu_{n-2})\in\mathbb{N}_0^{n-2}$ and monodromy maps $$\sigma_\nu^M,\sigma_\nu^N\colon\mathop{\natural}_{i=1}^{n-2}\mathop{\natural}_{\nu_i}{(S^{i}\times D^{n-1-i})}\rightarrow\mathop{\natural}_{i=1}^{n-2}\mathop{\natural}_{\nu_i}{(S^{i}\times D^{n-1-i})}$$ such that 
    $$(\mathop{\natural}_{i=1}^{n-2}\mathop{\natural}_{\nu_i}(S^{i}\times D^{n-1-i}),\sigma_\nu^M)\text{ and }(\mathop{\natural}_{i=1}^{n-2}\mathop{\natural}_{\nu_i}(S^{i}\times D^{n-1-i}),\sigma_\nu^N)$$ are open book decompositions on $X$ and $Y$, respectively. In particular, the open book decompositions on $X$ and $Y$ with trivial monodromy can be stabilized to open book decompositions with a common page.
\end{corollary}


\subsection*{Convention}
Unless otherwise stated, manifolds are connected, compact, oriented, and smooth, and the maps between them are smooth. We denote the manifold $X$ equipped with the opposite orientation by $\overline{X}$, and write $\cong$ to indicate the existence of an orientation-preserving diffeomorphism between smooth manifolds. Finally, $M \mathop{\natural} N$ denotes the boundary connected sum of two manifolds $M$ and $N$ with non-empty boundary, which is an oriented, connected manifold obtained by attaching an \handle{1} to the boundary of $M$ and $N$, and $\#$ denotes connected sum.

\subsection*{Acknowledgment}
I am grateful to my advisor Marc Kegel for his careful guidance throughout the development of this work. I am indebted to Livio Ferretti, Shital Lawande, Gheehyun Nahm, Qiuyu Ren, Kuldeep Saha, Felix Schm\"aschke, David Suchodoll, and Alison Tatsuoka for valuable discussions. I also thank the University of Heidelberg for providing a productive environment during part of the writing process.

The author is supported by the Claussen-Simon-Stiftung and is a member of the Berlin Mathematics Research Center MATH+ (EXC-2046/1, project ID: 390685689), funded by the Deutsche Forschungsgemeinschaft (DFG) under Germany’s Excellence Strategy.

\section{Preliminaries}
The main goal of this section is to prepare the readers for Definition~\ref{def induced handle decomposition}, which will be introduced in Section~\ref{sec handle decomposition}.
\subsection{Handle decompositions}
We begin by revisiting the definition and key properties of handle decompositions. Let $D^k$ and $S^{k-1}$ denote the $k$-dimensional disk and its boundary, the $(k-1)$-dimensional sphere, respectively.
\begin{definition}
    For integers $0\leq k\leq m$, an $m$-dimensional \textit{$k$-handle $\operatorname{h}^k$} of \textit{index} $k$ is a copy of $D^{k}\times D^{m-k}$ attached to an $m$-manifold $X$ along $\partial D^k\times D^{m-k}$ via an embedding
    \[
    \phi^k\colon\partial D^k\times D^{m-k}\rightarrow \partial X.
    \]
    The result of attaching $\operatorname{h}^k$ to $X$ via $\phi^k$ is denoted by $X\cup_{\phi^k} \operatorname{h}^k$. $D^k \times0$ and $0\times D^{n-k}$ are called the \textit{core} $c(\operatorname{h}^k)$ and \textit{cocore} $coc(\operatorname{h}^k)$, respectively. $\partial D^k \times 0$ (and its image $\phi^k(\partial D^k \times 0)$) and $0 \times \partial D^{n-k}$ are called the \textit{attaching sphere} $a(\operatorname{h}^k)$ and \textit{belt sphere} $b(\operatorname{h}^k)$, respectively. The domain (and its image) of $\phi^k$ is called the \textit{attaching region}.
\end{definition}

\begin{definition}
    Let $X$ be an $m$-manifold with boundary $\partial X=\partial_{+}X\sqcup\overline{\partial_{-}X}$.
    A \textit{relative handle decomposition} $\operatorname{h}$ on $(X,\partial_{-}X)$ is a filtration
    $$[0,1]\times\partial_{-}X=X_{-1}\subseteq X_0\subseteq X_1\subseteq \dots\subseteq X_L=X$$
    of finite length $L$ with the property: for all $k>-1$ there are $\mu_k\geq0$ \textit{attaching maps} $$\phi_j^{k}\colon \partial D^k\times D^{m-k}\rightarrow \partial X_{k-1}, j\in\{1,\dots,\mu_k\},$$ such that
        $$X_k\cong X_{k-1}\cup_{\phi^{k}_1 }\operatorname{h}^k_1\cup_{\phi^{k}_2} \operatorname{h}^k_2 \dots\cup_{\phi^{k}_{\mu_k}} \operatorname{h}^k_{\mu_k}.$$
    A \textit{handle decomposition} on $X$ is a relative handle decomposition on $(X,\emptyset)$.
\end{definition}

One can always assume that handles are attached in order of non-decreasing indices. Furthermore, an $m$-dimensional manifold with a non-empty boundary always admits a handle decomposition with exactly one \handle{0} and no \handles{m}~\cite[Propositions~4.2.7 and~4.2.13]{gs}.

\begin{proposition}\label{dual handle decomposition}
    A handle decomposition on $X$ induces a \textit{dual} handle decomposition, a relative handle decomposition on $(X,\overline{\partial X})$. 
\end{proposition}
\begin{proof}[Proof outline]
    We sketch an argument presented in~\cite{gs}. Given a handle decomposition on $X$, we obtain a filtration starting with $[0,1]\times \partial X$ by reversing the order of handle attachment and interpreting each \handle{k} $D^k\times D^{m-k}$ as an \handle{(m-k)} $D^{m-k}\times D^{k}$, for all $k=1,\dots,m-1$.
\end{proof}

\begin{definition}
    Given an $m$-manifold with boundary $\partial X$ and an orientation-reversing diffeomorphism $\varphi\colon \partial X\rightarrow \overline{\partial X}$, let $D_\varphi X$ denote a closed $m$-manifold, oriented with the orientation matching $X$, obtained by gluing $X$ and $\overline{X}$ with $\varphi$. If $\varphi=\operatorname{id}$ is the identity, we omit the gluing map from the notation.
\end{definition}

For example, if $X$ is the closed unit interval, then its double $D X$, obtained by gluing two copies of $[0,1]$ along their boundaries, is diffeomorphic to $S^1$.

\begin{proposition}\label{handle decomposition on double}
    A handle decomposition on $X$ induces a natural handle decomposition on $D_\varphi X$.
\end{proposition}
\begin{proof}
    A handle decomposition on $D_{\varphi}X$ is obtained by adding handles to $X$ along $\partial X$. Given an orientation-preserving diffeomorphism $\varphi\colon\partial X\rightarrow\partial X$, consider $\varphi\colon\partial X\rightarrow \overline{\partial X}$ as an orientation-reversing diffeomorphism by reversing the orientation of the target. Pick a collar neighborhood $[0, 1]\times\partial X$ and identify it with $X_{-1}=[0,1]\times \overline{\partial X}$ of the dual handle decomposition on $(X,\overline{\partial X})$ using $\operatorname{id}_{[0,1]}\times \varphi$.
\end{proof}

\subsection{Half open books}
As an intermediate step in constructing a handle decomposition on an open book, we recall the notion of half open books~\cite{hsueh2023kirby}. By applying Proposition~\ref{handle decomposition on double}, we obtain a handle decomposition for the open book from a handle decomposition on the half open book.

Let $M$ be a \textit{page}, an $(n-1)$-manifold with a non-empty boundary $\partial M$, and let $\varphi$ be a \textit{monodromy map} on $M$, a self-diffeomorphism that fixes $\partial M$ pointwise. Consider the following two equivalence relations on the $n$-manifold $M\times[a,b]$, where $a<b\in\mathbb{R}$:
\begin{enumerate}
    \item\label{0.5} $(x,t)\sim_1(x,t') \text{ for all } x\in \partial M, t,t'\in [a,b]$ and
    \item\label{1} $(x,b)\sim_2(\varphi(x),a)\text{ for all } x\in M.$
\end{enumerate}

An $n$-manifold is an \textit{open book} $\operatorname{Ob}(M,\varphi)$ with page $M$ and monodromy $\varphi$, if it is diffeomorphic to the quotient $M\times [0,1]/(\sim_{\ref{0.5}}\text{and}\sim_{\ref{1}})$ by the equivalence relations~\ref{0.5} and~\ref{1} above.

\begin{definition}
    An $n$-manifold is a \textit{half open book} $\operatorname{hob}(M)$ with page $M$, if it is diffeomorphic to the quotient $M\times [0,\frac{1}{2}]/\sim_{\ref{0.5}}$ by the equivalence relation~\ref{0.5} above. $M\times \{0\},M\times \{\frac{1}{2}\} \subset \partial\operatorname{hob}(M)$ are called the \textit{front} and \textit{back cover} of the half open book, respectively.
\end{definition}

\begin{example}
    The half open book with page $[0,1]$, as depicted in Figure~\ref{fig:hob[0,1]}, is obtained by identifying $(0,0)\in [0,1]\times [0,\frac{1}{2}]$ with $(0,t)$ for all $t\in[0,\frac{1}{2}]$ and identifying $(1,0)\in [0,1]\times [0,\frac{1}{2}]$ with $(1,t)$ for all $t\in[0,\frac{1}{2}]$. $\operatorname{Hob}([0,1])$ is diffeomorphic to $D^2$ and its boundary $S^1$ can be seen as two intervals identified along the boundary.
\end{example}
\begin{figure}[ht]
    \centering
    \includegraphics[width=0.85\textwidth]{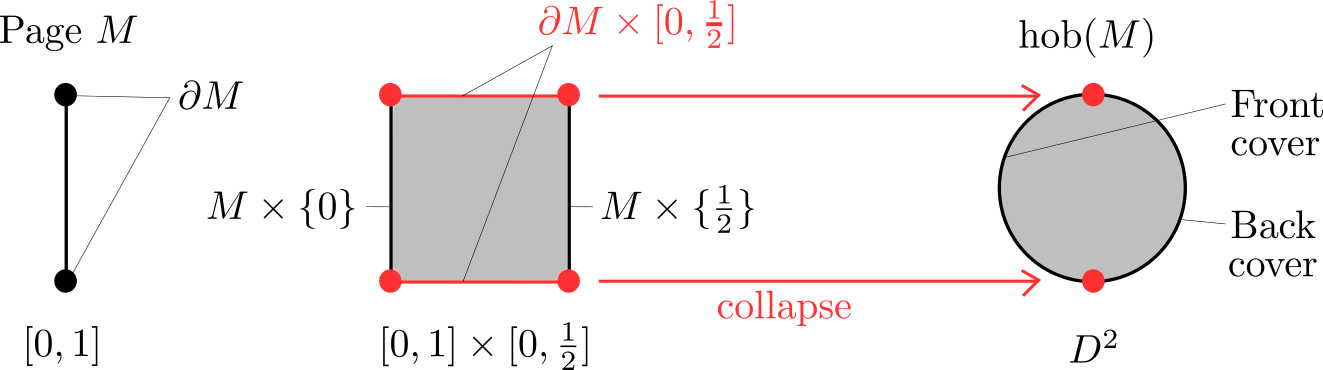}
    \caption{Constructing the half open book with page $[0,1]$.}
    \label{fig:hob[0,1]}
\end{figure}

Observe that the boundary of a half open book is given by the double $D M$ of the page $M$. More specifically, it is a closed, oriented manifold obtained from the front cover $M\times\{0\}$ by gluing the back cover $M\times\{\frac{1}{2}\}$ to it using the identity map along $\partial M$.

\begin{lemma}{\cite[Proposition~3.6]{hsueh2023kirby}}\label{lemma handle decomposition on hob}
     A handle decomposition on $M$ induces a natural handle decomposition on $\operatorname{hob}(M)$.
\end{lemma}
\begin{proof}[Proof outline]
    Starting with a handle decomposition on $M$, one constructs a handle decomposition on $M \times [0, \frac{1}{2}]$. Each $(n-1)$-dimensional \handle{k} of $M$ induces an $n$-dimensional \handle{k} in $M \times [0, \frac{1}{2}]$. By an inductive argument, the quotient by $\sim_{\ref{0.5}}$ preserves the handle decomposition.
\end{proof}

The key observation is that $D^{(n-1)-k}\times [0,\frac{1}{2}]/\sim_{\ref{0.5}}$ is diffeomorphic to $D^{n-k}$; see Figure~\ref{fig:handle_decomposition_on_hob} for an example. 
\begin{figure}[ht]
    \centering
    \includegraphics[width=0.7\textwidth]{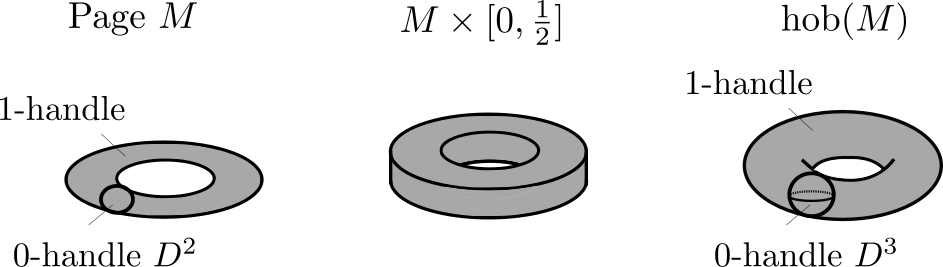}
    \caption{The solid torus is obtained from the annulus times half-interval by collapsing $\{p\}\times[0,\frac{1}{2}]$ to a point for all points $p$ in the boundary of the annulus.}
    \label{fig:handle_decomposition_on_hob}
\end{figure}

\section{A symmetric handle decomposition on open books}\label{sec handle decomposition}
This section aims to construct a handle decomposition on an open book $\operatorname{Ob}(M,\varphi)$ from a handle decomposition on the page $M$.

\begin{definition}
    Given a monodromy map $\varphi \colon M \rightarrow M$ on a page $M$, we extend it to a map ${\varphi}^+ \colon D M \rightarrow D M$ on the double by applying the identity map on the second copy of $M$.
\end{definition}

In particular, the open book $\operatorname{Ob}(M,\varphi)$ is diffeomorphic to the result $D_{\varphi^{+}}\operatorname{hob}(M)$ of gluing two half open books.


    

\begin{definition}\label{def induced handle decomposition}
    A handle decomposition on $\operatorname{Ob}(M,\varphi)$ is said to be \textit{induced} by a handle decomposition $\operatorname{h}$ on $M$ and denoted by $\mathbf{h}\cup_{\varphi}\mathbf{h}^*$, if it is obtained as follows.
\begin{enumerate}
    \item Use Lemma~\ref{lemma handle decomposition on hob} to obtain a handle decomposition $\mathbf{h}$ on $\operatorname{hob}(M)$ from $\operatorname{h}$.

    \item Apply Proposition~\ref{dual handle decomposition} to the handle decomposition obtained in the previous step to obtain a dual handle decomposition $\mathbf{h}^{*}$ on $(\operatorname{hob}(M),\overline{D M})$.
    
    \item Glue the above obtained handle decompositions together to give a handle decomposition on $D_{\varphi^{+}}\operatorname{hob}(M)$ as in Proposition~\ref{handle decomposition on double}.
\end{enumerate}
\end{definition}

\begin{definition}
    Let $M$ be an $(n-1)$-dimensional manifold with boundary.
    A handle decomposition on $M$ is an element of $\mathcal{H}{(\mu_1,\mu_2,\dots,\mu_{n-2})}$ if it consists of one \handle{0}, $\mu_1$ \handles{1}, $\dots$, $\mu_{n-2}$ \handles{(n-2)}, and no \handles{(n-1)} if its boundary is non-empty or one \handle{(n-1)} if its boundary is empty.
    
    Similarly, we say a relative handle decomposition on $(M,\partial M)$ is an element of $\mathcal{H}^*{(\lambda_1,\lambda_2,\dots,\lambda_{n-2})}$ if it consists of no \handle{0}, $\lambda_1$ \handles{1}, $\lambda_2$ \handles{2}, $\dots$, $\lambda_{n-2}$ \handles{(n-2)}, and one $\handle{($n-1$)}$.
\end{definition}

\begin{table}[ht]
    \centering

    \begin{tabular}{c|c l}

   Definition~\ref{def induced handle decomposition} & Handle decomposition &  on\\
    \hline
      Input  & $\operatorname{h}\in\mathcal{H}{(\mu_1,\mu_2,\dots,\mu_{n-2})}$ & $M$ \\
      
      Step~1 & $\mathbf{h}\in\mathcal{H}{(\mu_1,\mu_2,\dots,\mu_{n-2},0)}$ & $\operatorname{hob}(M)$ \\
      
      Step~2 & $\mathbf{h}^*\in\mathcal{H}^*{(0,\mu_{n-2},\dots,\mu_2,\mu_1)}$ & $(\operatorname{hob}(M),\overline{D M})$ \\

      Output  & $\mathbf{h}\cup_{\varphi}\mathbf{h}^*\in\mathcal{H}{(\mu_1,\mu_2+\mu_{n-2},\dots,\mu_{n-2}+\mu_2,\mu_1)}$ & $\operatorname{Ob}(M,\varphi)$
  
    \end{tabular}
    \caption{Number of handles of an induced handle decomposition.}
    \label{tab:number of handles}
\end{table}

\begin{remark}
Table~\ref{tab:number of handles} reflects the well known fact that the Euler characteristic of an open book is twice the Euler characteristic of its page if the page is odd-dimensional and zero otherwise. Let $M$ be an $(n-1)$-dimensional page with a handle decomposition $\operatorname{h}\in\mathcal{H}{(\mu_1,\mu_2,\dots,\mu_{n-2})}$. Since the Euler characteristic $\chi(\operatorname{Ob}(M,\varphi))$ is given by the alternating sum of the number of handles, therefore $\chi(\operatorname{Ob}(M,\varphi))=$
    \begin{equation*}
      \sum_{i=0}^{n-2}(-1)^{i}\mu_i+\sum_{i=0}^{n-2}(-1)^{n-i}\mu_i=
        \begin{cases}
          2\sum_{i=0}^{n-2}(-1)^{i}\mu_i=2\chi(M)  & \text{if $n-1$ is odd, and}\\
          0 & \text{if $n-1$ is even.}
        \end{cases}
    \end{equation*}
\end{remark}


We clarify our notation before presenting the next lemma on induced handle decompositions. Given a handle decomposition $\operatorname{h}$, we assign an order to the handles and let $\operatorname{h}_j^k$ denote the $j$-th \handle{k}. Throughout this paper, $\operatorname{h}_j^k$ refers to an $(n-1)$-dimensional handle of a page, while $\mathbf{h}_j^k$ denotes an $n$-dimensional handle of a (half) open book. Let $\phi_j^k$ and $\Phi_j^k$ represent their attaching maps, respectively. We use a $^*$ to indicate the dual of a handle, and the attaching map of a dual handle ${\operatorname{h}_j^k}^*$ is denoted by ${\varphi_j^k}^*$.

\begin{lemma}\label{lemma attaching region}
    Let $\mathbf{h}\cup_{\varphi}\mathbf{h}^*$ be a handle decomposition on $\operatorname{Ob}(M,\varphi)$ induced by $\operatorname{h}$, then the attaching sphere $$a({\mathbf{h}_j^k}^{*})=coc(\operatorname{h}_j^k)\cup \varphi(coc(\operatorname{h}_j^k))\subset \partial \operatorname{hob}(M)=D M$$ is the union of the cocore $coc(\operatorname{h}_j^k)\subset M\times\{\frac{1}{2}\}$ of $\operatorname{h}_j^k$ in the back cover and its image $\varphi(coc(\operatorname{h}_j^k))\subset M\times\{0\}$ in the front cover. Moreover, the attaching region $${\Phi_j^k}^*(\partial D^{n-k}\times D^{k})=\operatorname{h}_j^k\cup\varphi(\operatorname{h}_j^k)\subset D M$$ is the union of $\operatorname{im}({\Phi_j^k}^*)\cap M\times\{\frac{1}{2}\}=\operatorname{h}_j^k$ in the back cover and $\varphi(\operatorname{h}_j^k)=\operatorname{im}({\Phi_j^k}^*)\cap M\times\{0\}$ in the front cover.
\end{lemma}
\begin{proof}
    The handle ${\mathbf{h}_j^k}^*$ is an $n$-dimensional $(n-k)$-handle, so its attaching sphere is $(n-k-1)$-dimensional. The cocore of the $(n-1)$-dimensional $k$-handle $\operatorname{h}_j^k$ is an $(n-k-1)$-dimensional disk, which confirms that the dimensions are consistent.
    
    Recall that $\mathbf{h}_j^k\subset \operatorname{hob}(M)$ is induced by $\operatorname{h}_j^k\subset M$. Specifically, $\mathbf{h}_j^k$ is the quotient of $\operatorname{h}_j^k\times[0,\frac{1}{2}]\cong D^{k}\times D^{(n-1)-k}\times[0,\frac{1}{2}]$ under $\sim_{\ref{0.5}}\text{and}\sim_{\ref{1}}$. In the case of trivial monodromy, the attaching sphere ${a(\mathbf{h}_j^k}^{*})$ is isotopic to the belt sphere $b(\mathbf{h}_j^k)$. The belt sphere
    $$b(\mathbf{h}_j^k)=\partial \left(D^{(n-1)-k}\times [0,1/2]\right)\Big/\sim_{\ref{0.5}}\text{and}\sim_{\ref{1}}$$
    $$=\left(\partial D^{(n-1)-k}\times[0,1/2]\cup D^{(n-1)-k}\times\{0\}\cup D^{(n-1)-k}\times\{1/2\}\right)\Big/\sim_{\ref{0.5}}\text{and}\sim_{\ref{1}}$$
    is isotopic to
    $$\left(D^{(n-1)-k}\times\{0\}\big/\sim_{\ref{0.5}}\text{and}\sim_{\ref{1}}\right)\cup\left( D^{(n-1)-k}\times\{1/2\}\big/\sim_{\ref{0.5}}\text{and}\sim_{\ref{1}}\right)$$
    $$=coc(\operatorname{h}_j^k)\cup coc(\operatorname{h}_j^k)$$
    the union of the cocores of $\operatorname{h}_j^k$ in the front and back covers as shown in Figure~\ref{fig:union of cocore}.
    
\begin{figure}
    \centering
    \includegraphics[width=0.75\textwidth]{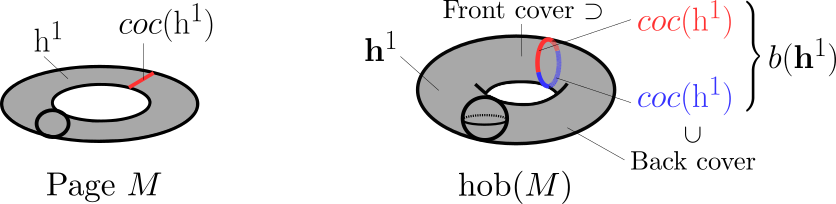}
    \caption{Trivial monodromy case: $a({\mathbf{h}^1}^{*})$ and $b(\mathbf{h}^1)$ are isotopic to $coc(\operatorname{h}^1)\cup coc(\operatorname{h}^1)\subset \partial\operatorname{hob}(M)\cong M\times\{0\}\cup_{\partial M}M\times\{\frac{1}{2}\}$.}
    \label{fig:union of cocore}
\end{figure}

    For the non-trivial monodromy case, we still use the identity map to glue the back cover of $\operatorname{hob}(M)$ to the front cover of $(\operatorname{hob}(M),\overline{D M})$. Thus, the hemisphere of the attaching sphere $a({\mathbf{h}_j^k}^{*})$ lying inside the back cover agrees with $coc(\operatorname{h}_j^k)\subset M\times\{\frac{1}{2}\}$ as in the trivial monodromy case. Now we use the monodromy map to glue the back cover of $(\operatorname{hob}(M),\overline{D M})$ to the front cover of $\operatorname{hob}(M)$, therefore the hemisphere of the attaching sphere $a({\mathbf{h}_j^k}^{*})$ in the front cover is given by $\varphi(coc(\operatorname{h}_j^k))$ instead of $\operatorname{id}(coc(\operatorname{h}_j^k))$.

    Let $D_\pm^{n-k-1}$ denote the upper and lower hemispheres of $\partial D^{n-k}$. Since the gluing map on the back cover is given by $\operatorname{id}_{M\times\{\frac{1}{2}\}}$, thus $${\Phi_j^k}^*(D_+^{n-k-1}\times D^k)=D^k\times D^{n-k-1}=\operatorname{h}_j^k\subset M\times \{1/2\}.$$ In addition to interchanging the factors $D_-^{n-k-1}\times D^k$, we also consider the gluing map in the back cover given by $\varphi$. Therefore, $${\Phi_j^k}^*(D_-^{n-k-1}\times D^k)=\varphi(D^k\times D^{n-k-1})=\varphi(\operatorname{h}_j^k)\subset M\times \{0\}. \eqno\qed$$
\renewcommand{\qed}{}
\end{proof}

\begin{example}\label{ex dehn twist}
    We construct a handle decomposition on $3$-dimensional open books with page $S^1\times D^1$. As seen in Figure~\ref{fig:handle_decomposition_on_hob}, the annulus and the half open book with annulus page both admit a handle decomposition consisting of one \handle{0} and one \handle{1}. Consider the monodromy $\tau$ given by a single Dehn twist along $S^1\times\{0\}\subset S^1\times D^1$, which sends the dashed arc in Figure~\ref{fig:part b} to the arc in Figure~\ref{fig:image of cocore}. The attaching region of the induced \handle{2} is depicted in Figure~\ref{fig:attaching sphere}. This \handle{2} is in canceling position with the \handle{1} as its attaching sphere, given by the union of the coccore of the \handle{1} in the back cover with the image $\tau(coc(D^1\times D^1))$ in the front cover, intersects the belt sphere of the \handle{1} geometrically once. We obtain a handle decomposition of the open book $(S^1 \times D^1, \tau)$ by adding the induced \handle{3}. This is a handle decomposition of $S^3$ because after canceling handles, we recover $S^3$ as the union of a single \handle{0} and a single \handle{3}. Forming open book connected sum with this open book decomposition of $S^3$ is called a \textit{Hopf plumbing}~\cite{Harer}. The binding of this open book is a Hopf link, hence the name, see~\cite{Etnyre} for more details.
    \begin{figure}[ht]
        \centering
            \begin{subfigure}{0.55\textwidth}
                \centering
                \includegraphics[scale=0.5]{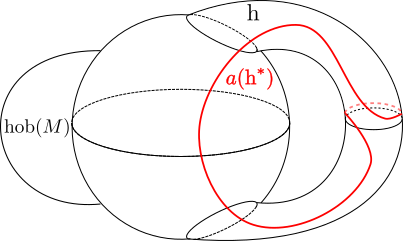}
            \caption{$a(\mathbf{h}^*)=coc(\operatorname{h})\cup\tau(coc(\operatorname{h}))$.}
            \label{fig:attaching sphere}
            \end{subfigure}
        \hfill
            \begin{subfigure}{0.44\textwidth}
            
                \centering
                \includegraphics[scale=0.45]{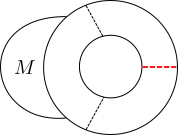}
                \caption{Back cover: $coc(\operatorname{h})$.}
                \label{fig:part b}

                \centering
                \includegraphics[scale=0.45]{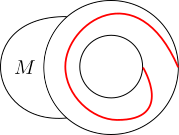}    \caption{Front cover: $\tau(coc(\operatorname{h}))$.}
                \label{fig:image of cocore}

            \end{subfigure}
        \caption{}
        \label{fig:stabilization3d}
    \end{figure}
\end{example}


\section{Handle exchange}\label{sec handle exchange}



\begin{definition}
    Given a handle decomposition $\operatorname{h}$ with one \handle{0}, $\mu_1$ \handles{1}, $\dots$, $\mu_{n-2}$ \handles{(n-2)}, and $\mu_{n-1}$ top-dimensional \handles{(n-1)}, a subset $$\mathcal{A}\subseteq\{(j,k)|2\leq k\leq n-1, 1\leq j\leq \mu_k \}\subset \{(j,k)|\operatorname{h}_j^k\text{ is a handle of $\operatorname{h}$}\}$$ is called a \textit{selection} if it satisfies the following: Suppose $(j,k)\in\mathcal{A}$ and there is a handle $\operatorname{h}_i^\ell$ whose attaching sphere cannot be isotoped to have no intersection with the belt sphere of $\operatorname{h}_j^k$, then $(i,\ell)\in\mathcal{A}$.
\end{definition}

\begin{example}
    Consider the handle decomposition on $D^n$ consisting of one \handle{0}, a pair of canceling $(k-1)$- and \handles{k}, where $3\leq k\leq n$. The intersection of the attaching sphere of the \handle{k} with the belt sphere of the \handle{(k-1)} cannot be removed by isotopy. Therefore, the \handle{(k-1)} does not form a selection, while the \handle{k} itself is a selection. The \handle{(k-1)} union the \handle{k} form the other possible selection of this given handle decomposition. Note that removing the $k$-handle from this handle decomposition is a handle decomposition of a well-defined manifold, while the result of removing the $(k-1)$-handle is not.
\end{example}

A selection leads to a handle decomposition of a potentially non-diffeomorphic manifold obtained by eliminating all the selected handles. We proceed by removing the selected handles in descending order of index. Suppose a \handle{k} is selected and no handle of index $>k$ is selected. We isotope the attaching sphere of any handle of indices $>k$ to have no intersection with the belt sphere of the selected \handle{k}, which is possible by definition, then the result of removing this handle is still a handle decomposition of a manifold. Recall that the diffeomorphism type of the resulting manifold is independent of the order in which handles of a given index are removed, just as the order does not matter when attaching handles of the same index. Inductively, the resulting handle decomposition with all selected handles removed determines a manifold up to diffeomorphism.


\begin{definition}\label{def exchanged page}
    Given a selection $\mathcal{A}$ on a handle decomposition $\operatorname{h}$ on $M$, the 
    \textit{exchanged} page, denoted by $\calA{M}$, is obtained from $M$ by replacing each selected handle $\operatorname{h}_j^k$, $(j,k)\in \mathcal{A}$, by a boundary connected sum $S^{n-k}\times D^{k-1}$ along the boundary component to which $\operatorname{h}_j^k$ was attached.
\end{definition}
The exchanged page $\calA{M}$ inherits a handle decomposition, denoted by $\calA{\operatorname{h}}$, from $\operatorname{h}$ on $M$, by regarding ${S}^{n-k}\times D^{k-1}$ as a \handle{(n-k)} attachment. We denote the handle decomposition on $\operatorname{hob}(\calA{M})$ induced by $\calA{\operatorname{h}}$ as $\calA{\mathbf{h}}$.


We now restrict our attention to selections that are compatible with a given monodromy map. In the special case where the monodromy is the identity on the page, no restrictions are imposed on the selection.

\begin{definition}\label{def exchangeable selection}
    Given a handle decomposition $\operatorname{h}$ on $M$ and a monodromy map $\varphi\colon M\rightarrow M$, a selection $\mathcal{A}$ is \textit{exchangeable} if the restriction of $\varphi$ on the selected handles is isotopic to the identity map, or trivial for short.
\end{definition}

We now describe the paths along which we would like to slide handles. By isotopy, we may always assume that $a(\mathbf{h}_j^k)\cap \partial D^n\neq\emptyset$, which means that the attaching sphere of each handle $\mathbf{h}_j^k$ (with index $k>0$) lies on the boundary of the \handle{0} at least partially. By duality, $b({\mathbf{h}_j^k}^*)\cap \partial D^n\neq\emptyset$, i.e.\ the belt sphere of each handle ${\mathbf{h}_j^k}^*$ (with index $k>0$) lies on the boundary of the top-dimensional \handle{n} at least partially.

\begin{definition}
    Let $\mathcal{A}$ be a selection on a handle decomposition $\operatorname{h}$ on $M$. For each $(j,k)\in\mathcal{A}$, let $\gamma^k_j\colon [0,1]\rightarrow c(\mathbf{h}_j^k)\cong D^k$ be a radial path from the center of the core of $\mathbf{h}_j^k$ to a point $\gamma^k_j(1)\in a(\mathbf{h}_j^k)\cap \partial D^{n}$ on the perimeter of the core disk, as illustrated in Figure~\ref{fig:exchange_move}. Similarly, let ${\gamma_j^k}^*\colon[0,1]\rightarrow c({\mathbf{h}_j^k}^*)\cong D^{n-k}$ be a radial path from a point ${\gamma_j^k}^*(0)\in a({\mathbf{h}_j^k}^*)\cap \partial D^n$ on the perimeter of the core disk to the center of the core of ${\mathbf{h}_j^k}^*$.
\end{definition}

Recall that by the definition of exchangeability, the monodromy restricts to the identity map on a selected handle ${\operatorname{h}_j^k}$, and the attaching sphere $a({\mathbf{h}_j^k}^*)$ is isotopic to the belt sphere $b(\mathbf{h}_j^k)$. Therefore, we can slide selected handles as follows, making the exchanged page come into the picture.

\begin{definition}\label{def A-modification of handle decomposition}
    Let $\mathbf{h}\cup_{\varphi}\mathbf{h}^*$ be an induced handle decomposition on $\operatorname{Ob}(M,\varphi)$ and $\mathcal{A}$ an exchangeable selection of $\operatorname{h}$. The following serie of handle slides on $\mathbf{h}\cup_{\varphi}\mathbf{h}^*$ are called \textit{exchange moves} and the resulting handle decomposition on $\operatorname{Ob}(M,\varphi)$ is denoted by $\calA{{(\mathbf{h}\cup_{\varphi}\mathbf{h}^*)}}$. For each $(j,k)\in\mathcal{A}$,
    \begin{enumerate}
        \item isotope $a({\mathbf{h}_j^k}^*)$ along $\gamma^k_j$, from $b(\mathbf{h}_j^k)=0\times\partial D^{n-k}$ to $\gamma^k_j(1)\times\partial D^{n-k}$, into the boundary of the \handle{0} (Figure~\ref{fig:exchange_move}), and
        \item isotope $a(\mathbf{h}_j^k)$ along ${\gamma_j^k}^*$, from ${\gamma_j^k}^*(0)\times\partial D^{k}$ to $b({\mathbf{h}_j^k}^*)=0\times\partial D^k$, which is equivalent to isotoping the belt sphere $b(\mathbf{h}_j^k)$ into the boundary of the \handle{n}.
    \end{enumerate}
\end{definition}

\begin{figure}
        \centering
        \includegraphics[width=0.6\textwidth]{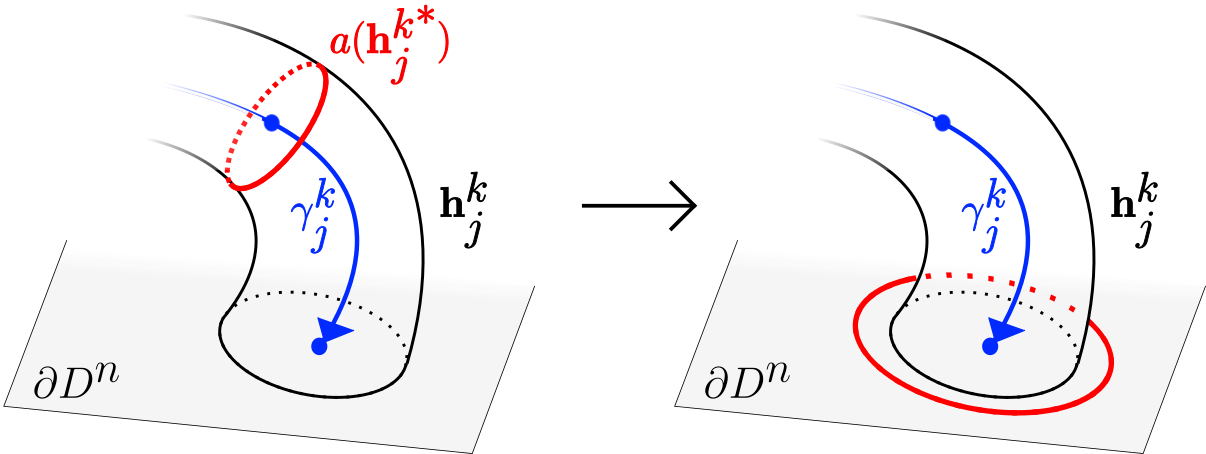}
        \caption{Isotoping the attaching sphere $a({\mathbf{h}_j^k}^*)$ of the dual of a selected handle $\mathbf{h}_j^k$ into the boundary of the \handle{0}.}
        \label{fig:exchange_move}
\end{figure}

Figure~\ref{fig:exchange_move_example} illustrates a schematic of a selected handle (blue) and its dual (red) performing an exchange move.

\begin{figure}[ht]
    \centering
    \includegraphics[width=0.64\linewidth]{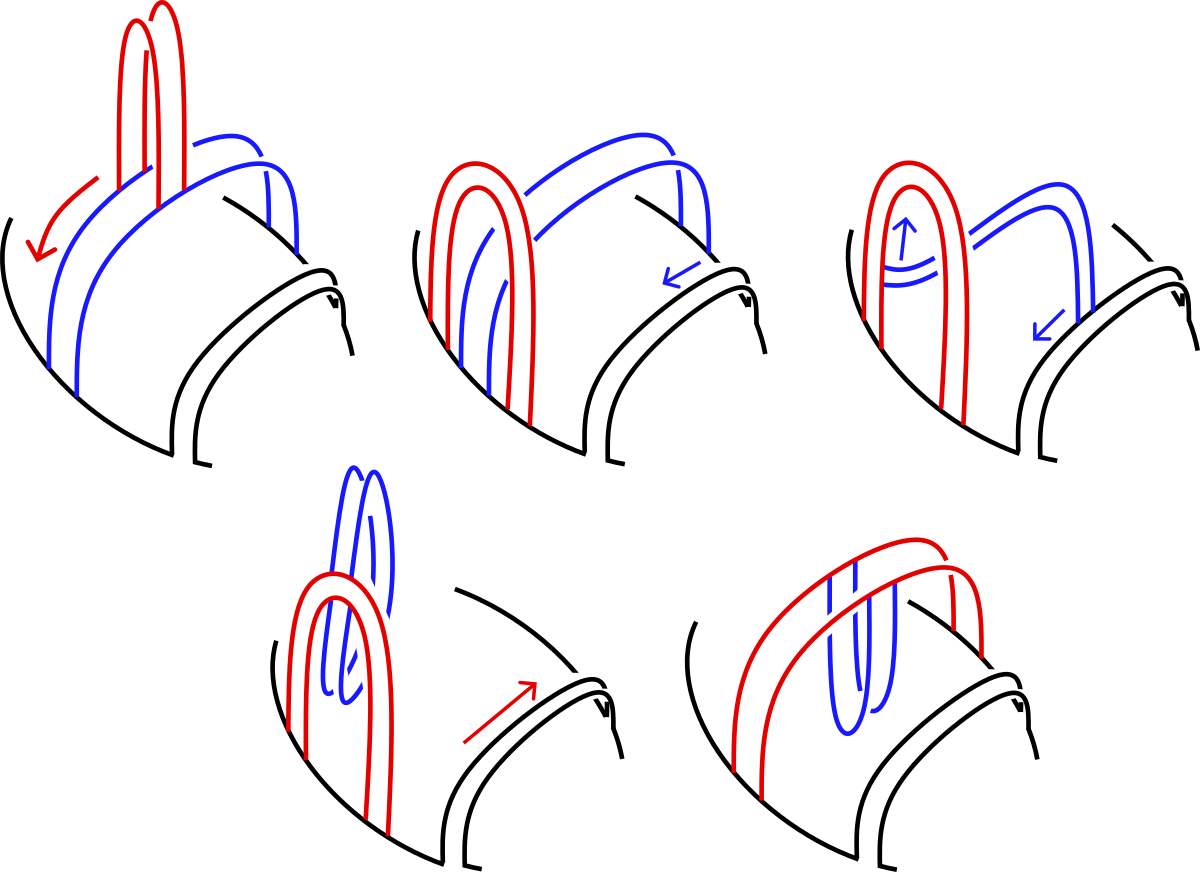}
    \caption{A schematic of exchange moves.}
    \label{fig:exchange_move_example}
\end{figure}

\begin{theorem}\label{thm technical}
    There exists a monodromy map $\calA{\varphi}\colon\calA{M}\rightarrow\calA{M}$ on the exchanged page $\calA{M}$ such that performing exchange moves on the handle decomposition on $\operatorname{Ob}(M,\varphi)$ induced by a handle decomposition $\operatorname{h}$ on $M$ gives the handle decomposition on $\operatorname{Ob}(\calA{M},\calA{\varphi})$ induced by the handle decomposition $\calA{\operatorname{h}}$ on the exchanged page $\calA{M}$, that is to say the diagram in Figure~\ref{fig:diagram} commutes.
    \begin{figure}
        \centering
        \begin{tikzpicture}
            \node[] (a) at (-1.5,1){$\operatorname{h}$};
            \node[] (b) at (6,1){$\calA{\operatorname{h}}$};
            \node[] (c) at (-1.5,0){${\mathbf{h}\cup_{\varphi}\mathbf{h}^*}$};
            \node[] (d) at (3.9,0){$\calA{{(\mathbf{h}\cup_{\varphi}\mathbf{h}^*)}}$};
            \node[] (e) at (6,0) {$=\calA{\mathbf{h}}\cup_{\calA{\varphi}}{\calA{\mathbf{h}}}^*$};
            \node[] at (2.2,1.2){\textit{page exchange}};
            \node[] at (1,0.2){\textit{exchange moves}};
            \graph{(a)->[](b)};
            \graph{(c)->[](d)};
            \graph{(a)->[](c)};
            \graph{(b)->[](e)};
        \end{tikzpicture} 
    \caption{The vertical arrows represent a handle decomposition of the page, inducing a handle decomposition of the open book.}
    \label{fig:diagram}
    \end{figure}
\end{theorem}
\begin{proof}
We start with a handle decomposition on $\operatorname{Ob}(M,\varphi)$ induced by a handle decomposition $\operatorname{h}$ on $M$. Since $\mathcal{A}$ is an exchangeable selection, for each $(j,k)\in\mathcal{A}$, ${\mathbf{h}_j^k}^*$ is a copy of $D^{n-k}\times D^{k}$, where $\partial D^{n-k}\times D^{k}$ is identified with $D^k\times \partial D^{n-k}\subset\mathbf{h}_j^k$. Firstly, the exchange moves slides the attaching region of ${\mathbf{h}_j^k}^*$ to the boundary of the \handle{0}, which allows us to view ${\mathbf{h}_j^k}^*$ as an \handle{(n-k)} attached to the \handle{0}. If we identify the \handle{0} $D^n$ with $D^{n-k}\times D^k$, then the attaching map of this \handle{(n-k)} can be isotoped to the identity map $\partial D^{n-k}\times D^k\rightarrow \partial D^{n-k}\times D^k$~\cite[Example~4.1.4.(d)]{gs}. Therefore, $\operatorname{hob}(M)\cup {\mathbf{h}_j^k}^*$ is diffeomorphic to $$\operatorname{hob}(M)\mathop{\natural} (S^{n-k}\times D^k),$$ where the boundary connected sum takes place along the boundary component that $\mathbf{h}_j^k$ is attached along. Note that $\operatorname{hob}(M)\mathop{\natural} (S^{n-k}\times D^{k})$ is diffeomorphic to $\operatorname{hob}(M\mathop{\natural} (S^{n-k}\times D^{k-1}))$. Secondly, the exchange moves involve sliding $\mathbf{h}_j^k$ such that it is no longer attached to the boundary of the \handle{0}, removing $\mathbf{h}_j^k$ from $\operatorname{hob}(M)$. Consequently, the union of the \handle{0}, all $\mu_1$ \handles{1}, the handles $\mathbf{h}_i^{\ell}$ with $(i,\ell)\notin \mathcal{A}$, and the dual handles ${\mathbf{h}_j^k}^*$ with $(j,k)\in\mathcal{A}$, coincides with a handle decomposition on the half open book with the exchanged page $M^{\mathcal{A}}$. Similarly, the remaining handles of $\calA{{(\mathbf{h}\cup_{\varphi}\mathbf{h}^*)}}$ --- the \handle{n}, the $\mu_1$ \handles{(n-1)}, the handles $\mathbf{h}_i^{\ell}$ with $(i,\ell)\in \mathcal{A}$, and the handles ${\mathbf{h}_j^k}^*$ with $(j,k)\notin\mathcal{A}$ --- form a relative handle decomposition on $(\operatorname{hob}(M^{\mathcal{A}}),\overline{D M^{\mathcal{A}}})$. In other words, $\calA{{(\mathbf{h}\cup_{\varphi}\mathbf{h}^*)}}$ is the union of a handle decomposition on $\operatorname{hob}(\calA{M})$ and a relative handle decomposition on $(\operatorname{hob}(M^{\mathcal{A}}),\overline{D M^{\mathcal{A}}})$. 
Suppose $\operatorname{h}_j^k$ was exchanged for $S^{n-k}\times D^{k-1}$, then the attaching map of the dual of $S^{n-k}\times D^{k}\subset\operatorname{hob}(\calA{M})$ is given by $({\Phi_j^k}^{*})^*=\Phi_j^k$. By isotopy, we can assume that the image under $\Phi_j^k$ of the upper and lower thickened hemispheres of its domain $\partial{D}^{k}\times D^{n-k}$ lie in the front and back covers, respectively, for each $(j,k)\in\mathcal{A}$. Moreover, $$\Phi_j^k({D}^{k-1}_+\times D^{n-k})=D^{n-k}\times D^{k-1}\subset S^{n-k}\times D^{k-1}\subseteq\calA{M}\times\{1/2\}$$ and $$\Phi_j^k({D}^{k-1}_-\times D^{n-k})\subseteq\calA{M}\times\{0\}.$$
Using Lemma~\ref{lemma attaching region}, we observe that the front cover of the second copy is glued to the back cover of the first with the identity map on $\calA{M}$. Let ${\calA{\varphi}}^+\colon D\calA{M}\rightarrow D\calA{M}$ denote the gluing map such that $$\calA{{(\mathbf{h}\cup_{\varphi}\mathbf{h}^*)}}=\operatorname{hob}(\calA{M})\cup_{{\calA{\varphi}}^+} \operatorname{hob}(\calA{M}),$$ then restriction of ${\calA{\varphi}}^+$ on the back cover gives the desired monodromy map $\calA{\varphi}\colon\calA{M}\rightarrow\calA{M}$. 
To conclude, the desired map $\calA{\varphi}$ on $\calA{M}$ is determined by the restriction of $\varphi$ on the handles $\operatorname{h}_i^\ell\subset\calA{M}$, $(i,\ell)\notin \mathcal{A}$, not involved in the exchange moves. On each $S^{n-k}\times D^{k-1}\subset\calA{M}$ summand (replacing $\operatorname{h}_j^k$), $\calA{\varphi}$ encodes the image of $\Phi_j^k$ in the front cover for each $(j,k)\in\mathcal{A}$.
\end{proof}

Since exchange moves are a sequence of handle slides and handle slides preserve diffeomorphism type~\cite{gs}, we obtain the following corollary.

\begin{corollary}\label{cor partially simplifying the page}
    Given a handle decomposition on $M$ and an exchangeable selection $\mathcal{A}$, there exists a monodromy map $\calA{\varphi}\colon\calA{M}\rightarrow\calA{M}$ on the exchanged page $\calA{M}$ such that $\operatorname{Ob}(M,\varphi)$ is diffeomorphic to $\operatorname{Ob}(\calA{M},\calA{\varphi})$. \qed
\end{corollary}

\begin{remark}\label{remark almost canonical page}
    Quinn proved that any open book of dimension $n$ admits an almost canonical page~\cite{quinn}, i.e.\ a page that admits a handle decomposition with no handles of indices greater than $\lceil\frac{n}{2}\rceil$. We recover this result for open books whose page admits a handle decomposition such that the selection $$\mathcal{S}=\{(j,k)|\lceil\frac{n}{2}\rceil<k\leq n-1,1\leq j\leq\mu_k\}$$ is exchangeable because the exchanged page $\calS{M}$ would have no handles of indices exceeding $\lceil\frac{n}{2}\rceil$. Notably, we can construct almost canonical pages for open books with trivial monodromy using handle exchange moves.
\end{remark}

\begin{example}\label{ex simplifying simple page}
Suppose $2\leq k\leq n-2$, then $S^{k}\times D^{n-1-k}$ admits a handle decomposition in $\mathcal{H}(0,\dots,0,\mu_k=1,0,\dots,0)$. Consider the identity map as monodromy, then the selection $\mathcal{A}=\{\operatorname{h}_1^k\}$ is exchangeable and the exchanged page is given by $\calA{{(S^{k}\times D^{n-1-k})}}={S^{n-k}\times D^{k-1}}$. The new monodromy map is given by $\calA{{\operatorname{id}_{S^k\times D^{n-1-k}}}}=\operatorname{id}_{S^{n-k}\times D^{k-1}}$.
Alternatively, 
$\operatorname{Ob}(S^k\times D^{n-1-k},\operatorname{id})$
$$\cong S^k\times\operatorname{Ob}(D^{n-1-k},\operatorname{id}) \cong\operatorname{Ob}(D^{k-1},\operatorname{id})\times S^{n-k}\cong\operatorname{Ob}({S}^{n-k}\times D^{k-1},\operatorname{id}).$$

In general, the boundary connected sum $\mathop{\natural}_{i=1}^{n-2}\mathop{\natural}_{\mu_i}(S^{i}\times D^{n-1-i})$ admits a handle decomposition in $\mathcal{H}(\mu_1,\mu_{2},\dots,\mu_{n-2})$. The selection $\mathcal{A}=\{(j,k)|2\leq k\leq n-2,\\ 1\leq j\leq \mu_k\}$ of all handles of indices greater than or equal to two is exchangeable with respect to the trivial monodromy map and the exchanged page is given by 
$$\calA{(\mathop{\natural}_{i=1}^{n-2}\mathop{\natural}_{\mu_i}(S^{i}\times D^{n-1-i}))}=\mathop{\natural}_{i=1}^{n-2}\mathop{\natural}_{\nu_i}(S^{i}\times D^{n-1-i}),$$ where $\mu_1=\nu_1$ and $\mu_i=\nu_{n-i}$ for all $i\in[2,n-2]$. Therefore, $$\operatorname{Ob}(\mathop{\natural}_{i=1}^{n-2}\mathop{\natural}_{\mu_i}(S^{i}\times D^{n-1-i}),\operatorname{id})\cong \operatorname{Ob}(\mathop{\natural}_{i=1}^{n-2}\mathop{\natural}_{\nu_i}(S^{i}\times D^{n-1-i}),\operatorname{id})\cong\#_{i=1}^{n-2}\#_{\mu_i} (S^{i}\times S^{n-i}).$$



\end{example}

\section{Main results and relation to previously known stabilizations}

\subsection*{Proof of Theorem~\ref{thm stabilization of open books}}\label{sec stabilization}
    Suppose $(M,\varphi)$ is an open book decomposition on a manifold $X$ of dimension $n$. Fix a natural number $k\in[2,n-1]$. Consider the boundary connected sum $M\mathop{\natural} (S^{k-1}\times D^{n-k})$ along a chosen connected component of $\partial M$. Let $U=D^{n-2}\subset\partial M$ denote part of the attaching region $\partial D^1\times D^{n-2}$ of the \handle{1} used for forming the boundary connected sum. Without loss of generality, $U$ is contained in the boundary of an $(n-1)$-ball $U'\subset M$, as shown in Figure~\ref{fig:subset_U}, such that the restriction $\varphi|_{U'}$ of the monodromy map on $U'$ is equal to $\operatorname{id}_{U'}$. 
    \begin{figure}
        \centering
        \includegraphics[scale=0.16]{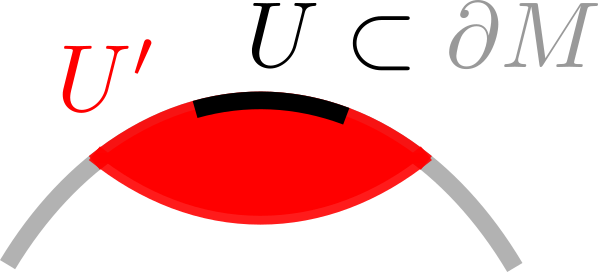}
        \caption{We perform boundary connected sum along $U$, which lies in a neighborhood of $\partial M$ where the monodromy is trivial.}
        \label{fig:subset_U}
    \end{figure}
    We regard $${S}^{k-1}\times D^{n-k}\cong {D^{k-1}\times D^{n-k}}{\cup_{\operatorname{id}_{\partial D^{k-1}\times D^{n-k}}} D^{k-1}\times D^{n-k}}$$ as a \handle{(k-1)} attached to $D^{k-1}\times D^{n-k}$. Let $\operatorname{h}^k$ be a \handle{k} that cancels with the \handle{(k-1)}, thus there exists a diffeomorphism $g$ from $D^{k-1}\times D^{n-k}$ to $(S^{k-1}\times D^{n-k})\cup \operatorname{h}^k$. Choose a diffeomorphism $f$ from $D^{n-1}$ to $D^{k-1}\times D^{n-k}$, then $M$ is diffeomorphic to $M\mathop{\natural} (S^{k-1}\times D^{n-k})\cup \operatorname{h}^k$ via a map $F$ that equals the identity outside of $U'$ and that is given by the composition $g\circ f$ on $U'$. Let $\varphi$ also denote the induced map on $M\mathop{\natural} (S^{k-1}\times D^{n-k})\cup \operatorname{h}^k$ given by $\varphi$ on $M-U'$ and identity elsewhere, then $F\circ \varphi=\varphi\circ F$. We now work with $(M\mathop{\natural} (S^{k-1}\times D^{n-k})\cup \operatorname{h}^k,\varphi)$ which is an open book decomposition on $X$ equivalent to the original. Since there are no handles attached along the belt sphere $b(\operatorname{h}^k)$ and the restriction of $\varphi$ is trivial on $\operatorname{h}^k$, $\mathcal{A}=\{\operatorname{h}^k\}$ is an exchangeable selection. By Corollary~\ref{cor partially simplifying the page}, there exists a monodromy map $\calA{\varphi}$ on $\calA{(M\mathop{\natural} (S^{k-1}\times D^{n-k})\cup \operatorname{h}^k)}$ such that $\operatorname{Ob}(\calA{(M\mathop{\natural} (S^{k-1}\times D^{n-k})\cup \operatorname{h}^k)},\calA{\varphi})$ is diffeomorphic to $\operatorname{Ob}(M\mathop{\natural} (S^{k-1}\times D^{n-k})\cup \operatorname{h}^k,\varphi)$. The exchanged page $$\calA{(M\mathop{\natural} (S^{k-1}\times D^{n-k})\cup \operatorname{h}^k)}$$ is given by $$M\mathop{\natural} (S^{k-1}\times D^{n-k})\mathop{\natural} (S^{n-k}\times D^{k-1}),$$ which is not diffeomorphic to $M$, providing an open book not equivalent to $(M,\varphi)$. \qed

\subsection*{Proof of Theorem~\ref{thm odd stabilization}}
    It suffices to show that $S^{2\ell+1}$ admits an open book decomposition of the form $(S^\ell\times D^\ell,\tau)$ for some monodromy map on $S^\ell\times D^\ell$.
    
    

    The page $S^{\ell}\times D^{\ell}$ can be considered as the result of attaching an \handle{\ell} $\operatorname{h}^\ell$ to $D^{2\ell}$ up to diffeomorphism, as described in the proof of Theorem~\ref{thm stabilization of open books}. Similarly, we view the half open book $\operatorname{hob}(S^\ell\times D^\ell)$ as the result of attaching an \handle{\ell} $\mathbf{h}^\ell$ to $\operatorname{hob}(D^{2\ell})\cong D^{2\ell+1}$.
    Let $\Phi\colon \partial D^{\ell+1}\times D^{\ell}\rightarrow \partial \operatorname{hob}( S^{\ell}\times D^{\ell})$ be the attaching map of an \handle{(\ell+1)} $\mathbf{h}^{\ell+1}$ that cancels with $\mathbf{h}^\ell$. Without loss of generality, we may assume that this $(2\ell+1)$-dimensional \handle{(\ell+1)} $\mathbf{h}^{\ell+1}$ is induced by a $(2\ell)$-dimensional \handle{(\ell+1)} $\operatorname{h}^{\ell+1}$ that forms a canceling pair with $\operatorname{h}^\ell$. Let ${D}^{\ell}_\pm\times D^{\ell}\subset \partial D^{\ell+1}\times D^{\ell}$ denote the upper and lower thickened hemispheres of its attaching region. By isotopy, $$\Phi({D}^{\ell}_+\times D^{\ell})=D^{\ell}\times D^{\ell}=\operatorname{h}^\ell\subseteq S^{\ell}\times D^{\ell}\times\{1/2\}$$ in the back cover and $$\Phi({D}^{\ell}_-\times D^{\ell})\subseteq S^{\ell}\times D^{\ell}\times\{0\}$$ in the front cover. 
    
    We obtain $S^{2\ell+1}$ by attaching a \handle{(2\ell+1)} along $\partial (\operatorname{hob}({S^\ell\times D^\ell})\cup_\Phi\mathbf{h}^{\ell+1})$, which is diffeomorphic to $\partial\operatorname{hob}(D^{2\ell})=S^{2\ell}$. The union of the \handle{(\ell+1)} and the \handle{(2\ell+1)} gives a relative handle decomposition on $(\operatorname{hob}(S^\ell\times D^\ell),\overline{\partial\operatorname{hob}(S^\ell\times D^\ell)})$. Therefore, $S^{2\ell+1}$ is the union of two half open books with page $S^\ell\times D^\ell$, and the front cover of the second copy is glued to the back cover of the first with the identity map on $S^\ell\times D^\ell$. Let ${\tau}^+\colon \partial\operatorname{hob}(S^\ell\times D^\ell)\rightarrow \partial\operatorname{hob}(S^\ell\times D^\ell)$ denote the gluing map such that $$S^{2\ell+1}=\operatorname{hob}(S^\ell\times D^\ell)\cup_{{\tau}^+}\operatorname{hob}(S^\ell\times D^\ell),$$
    then by Lemma~\ref{lemma attaching region} the restriction of ${\tau}^+$ on the back cover gives the desired monodromy map $\tau$. \qed

    

\begin{remark}\label{remark contact}
    A contact analog~\cite{Giroux, Seidel} of Theorem~\ref{thm odd stabilization} is well known. Our version is closely related to~\cite[Definition~2.6]{Saha}, which is a special case of~\cite[Theorem~4.6]{van_Koert}. A contact open book is an open book $(M,\varphi)$  with an exact symplectic page $M$ (say of dimension $2\ell$) and a boundary preserving symplectomorphism $\varphi$. The construction of a contact $1$-form on a contact open book can be found in~\cite{Geiges_book,TW}. Moreover, every contact manifold is \textit{supported by} an open book, i.e.\ is a contact open book~\cite{breen2024girouxcorrespondencearbitrarydimensions, contact, Giroux, licata2024girouxcorrespondencedimension3}. Given a contact open book $(M,\varphi)$ with the symplectic form on $M$ given by $\omega=d\lambda$, an $(\ell-1)$-sphere in $\partial M$ is said to be Legendrian if $\lambda$ vanishes on it. The result of attaching a Weinstein \handle{\ell} $\operatorname{h}_w$ to $M$ along a Lagrangian $(\ell-1)$-sphere $S$ carries a preferred symplectic structure~\cite{Weinstein}. If the core disk of the \handle{\ell} union a Lagrangian disk with boundary $S$ gives a $\ell$-sphere, then composing $\varphi$ with a Dehn--Seidel twist $\tau_{DS}$ along this $\ell$-sphere, produces a monodromy on $M\cup\operatorname{h}_w$ such that $(M,\varphi)$ and $(M\cup\operatorname{h}_w,{\varphi\circ\tau_{DS}})$ are contactomorphic contact open books. See~\cite{van_Koert} for relevant details. In our situation, we only attach $\ell$-handles to the \handle{0} of the page $M$ by an embedding of $\partial D^{\ell}\times D^{\ell}\rightarrow\partial D^{2\ell}$ that extend to an embedding $D^\ell\times D^\ell \rightarrow S^{2\ell-1}$ such that the result of the handle attachment is diffeomorphic to $M\mathop{\natural} (S^{\ell}\times D^{\ell})$ as in~\cite[Example~4.1.4.(b)]{gs}.
\end{remark}

\begin{remark}\label{remark on relation with other stabilizations}
The $2$-stabilization given in Theorem~\ref{thm stabilization of open books} corresponds to performing the Hopf plumbing twice. This can be seen through the sequence of isotopies of Heegaard diagrams as illustrated in Figure~\ref{fig:comparison_with_std_stabilization}; see also~\cite[Section~5.1]{hsueh2023kirby}. In these diagrams, each pair of identified disks denotes the attaching region of a \handle{1}, while each closed curve indicates the attaching sphere of a \handle{2}. The starting diagram represents $S^3$ and captures the $2$-stabilization of $(D^2,\operatorname{id})$ as the \handle{2} induced by the \handle{1} labeled A cancels with the \handle{1} labeled B, and vice versa. The final diagram depicts the connected sum of two copies of the open book with Hopf link binding. The two diagrams are isotopic, without any handle slides, and hence represent equivalent open books on $S^3$. The monodromy map of the open book on $S^3$ given by the $2$-stabilization of the standard open book on $S^3$ is called a barbell diffeomorphism in~\cite{budney_gabai_2}. 

\begin{figure}
    \centering
    \includegraphics[width=0.85\linewidth]{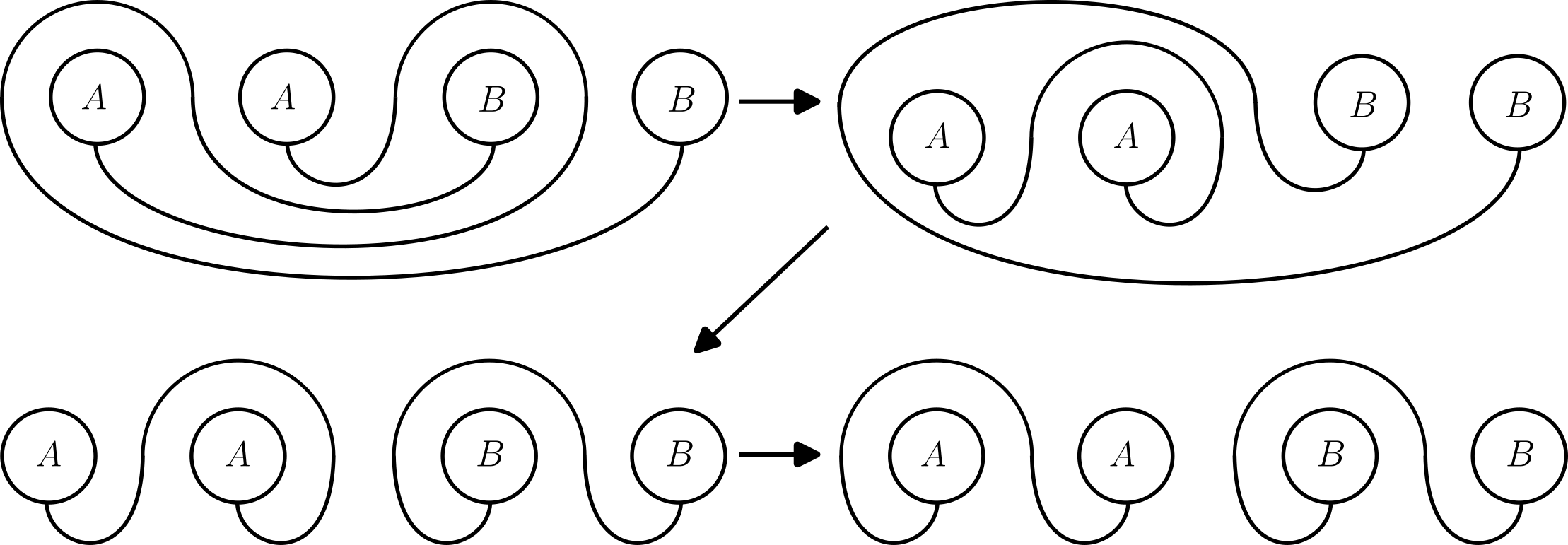}
    \caption{$2$-stabilization is Hopf plumbing twice. From the penultimate diagram to the final one: rotate the disk $A$ by $2\pi$.}
    \label{fig:comparison_with_std_stabilization}
\end{figure}
\end{remark}

\begin{remark}
    Open books on $S^5$ were intensively studied in~\cite{Saeki_1,Saeki_2,Saeki_3}. Saeki constructed an open book on $S^5$ with page $(S^2\times S^2)\#(S^2\times S^2)- D^4$, that he called a stabilizer, to show that an algebraically fibered simple 3-knot is stably geometrically fibered~\cite{Saeki_2}. This open book has binding $S^3$, thus cannot be obtained from $(D^4,\operatorname{id})$ by a single $2$-, $3$-, or $4$-stabilization defined in Theorem~\ref{thm stabilization of open books}.
\end{remark}


\subsection*{Proof of Theorem~\ref{thm book equivalent monodromies}}
    By applying Theorem~\ref{thm stabilization of open books} to the standard open book decomposition $(D^{n-1},\operatorname{id})$ on $S^n$, we obtain the open book decomposition $({(S^{k-1}\times D^{n-k})}\\ \mathop{\natural} {(S^{n-k}\times D^{k-1})},\tau_k)$ on $S^n$, not equivalent to $(D^{n-1},\operatorname{id})$, for each $n\geq3$ and $k\in[2,n-1]$. Suppose $k\in[2,\halff]$ and $n\geq4$. Let $$\tau_{k},\tau_{n-k+1}\colon  (S^{k-1}\times D^{n-k})\mathop{\natural} (S^{n-k}\times D^{k-1})\rightarrow (S^{k-1}\times D^{n-k})\mathop{\natural} (S^{n-k}\times D^{k-1})$$ be the monodromy maps obtained by Theorem~\ref{thm stabilization of open books}. To keep notation simple, let
    $$\tau_{k},\tau_{n-k+1}\colon (S^{k-1}\times S^{n-k})\#(S^{n-k}\times S^{k-1})\rightarrow (S^{k-1}\times S^{n-k})\#(S^{n-k}\times S^{k-1})$$
    also denote their trivial extension on the double $$D ( (S^{k-1}\times D^{n-k})\mathop{\natural} (S^{n-k}\times D^{k-1}))\cong(S^{k-1}\times S^{n-k})\#(S^{n-k}\times S^{k-1}).$$
    
    Consider the following induced maps on homology 
    $$\mathrm{H}_{k-1}(\tau_{k}),
    \mathrm{H}_{n-k}(\tau_{k}),
    \mathrm{H}_{k-1}(\tau_{n-k+1}),
    \mathrm{H}_{n-k}(\tau_{n-k+1})\colon\mathbb{Z}\times\mathbb{Z}\rightarrow\mathbb{Z}\times\mathbb{Z}.$$
    It follows from the two claims below that $\tau_{k}$ and $\tau_{n-k+1}$ are non-isotopic. Claims:
    \begin{enumerate}
        \item $\mathrm{H}_{k-1}(\tau_{k})$ is non-trivial, whereas $\mathrm{H}_{n-k}(\tau_{k})$ is trivial.
        \item $\mathrm{H}_{n-k}(\tau_{n-k+1})$ is non-trivial, whereas $\mathrm{H}_{k-1}(\tau_{n-k+1})$ is trivial.
    \end{enumerate} Since the proofs of Claims 1 and 2 are similar, we omit the latter.

    We like to represent generators of the relevant homology groups of the double in terms of submanifolds in $(S^{k-1}\times D^{n-k})\mathop{\natural} (S^{n-k}\times D^{k-1})$. See Figure~\ref{fig:generators} for a schematic of the situation. We can generate 
    $$\mathrm{H}_{k-1}((S^{k-1}\times S^{n-k})\#(S^{n-k}\times S^{k-1}))$$ by $\lambda=S^{k-1}\times\{\text{pt}\}$ and $\mu$ the double of $\{\text{pt}\}\times D^{k-1}$. Similarly, $$\mathrm{H}_{n-k}((S^{k-1}\times S^{n-k})\# (S^{n-k}\times S^{k-1}))$$ is generated by $\alpha$ the double of $\{\text{pt}\}\times D^{n-k}$ and $\beta=S^{n-k}\times\{q\}$.
    \begin{figure}
        \centering
        \includegraphics[width=0.5\linewidth]{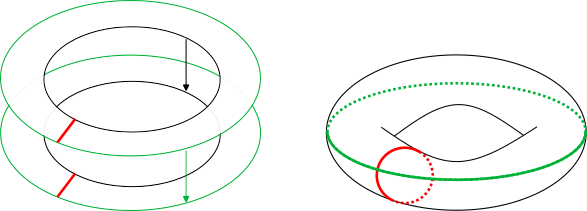}
        \caption{Generators of $\mathrm{H}_1D(S^1\times D^1)$ are given by the double of $\{\text{pt}\}\times D^1$ (red) and $S^1\times q$ (green) $\subset S^1\times D^1$, where $q\in \partial D^1$.}  
        \label{fig:generators}
    \end{figure}

    $\tau_{k}$ maps $\lambda$ to $\lambda$, since $\tau_{k}$ is trivial on $S^{k-1}\times D^{n-k}$, and $\mu$ to $\lambda\pm\mu$, which is isotopic to the belt sphere of the \handle{(k-1)} induced by $S^{k-1}\times D^{n-k}$. Thus $\mathrm{H}_{k-1}(\tau_{k})$ is non-trivial. It remains to show that $\mathrm{H}_{n-k}(\tau_{k})$ is the identity. $\tau_{k}$ sends $\alpha$ to $\alpha$ again because it is trivial on $S^{k-1}\times D^{n-k}$. Since $\tau_{k}$ is trivial near $S^{n-k}\times\partial D^{k-1}$, therefore $\mathrm{H}_{n-k}(\tau_{k})(\beta)=\beta$.\qed

\begin{remark}
    Infinitely many pairs of open book decompositions on the $3$-sphere having non-isotopic monodromy maps on a common page can be derived from~\cite{etnyre_li}. Etnyre and Li constructed a monodromy map $\phi_n$ on $\mathop{\natural}_{n-1} (S^1\times D^1$), such that $(\mathop{\natural}_{n-1} (S^1\times D^1),\phi_n)$ is an open book on $S^3$ for each $n\geq5$, where $\phi_n$ is a composition of Dehn twists along curves shown on the left of Figure~\ref{fig:planar_monodromy} (Figure~1 in~\cite{etnyre_li}). Let $\varphi_n$ denote the monodromy map on $\mathop{\natural}_{n-1} (S^1\times D^1)$ given by the composition of Dehn twists along the curves given to the right of Figure~\ref{fig:planar_monodromy}, then $(\mathop{\natural}_{n-1} (S^1\times D^1),\varphi_n)$ is another open book on $S^3$. It is shown in~\cite{etnyre_li} that $\phi_n$ has essential translation distances $dist_e(\phi_n)=2$, and since $(\mathop{\natural}_{n-1}(S^1\times D^1),\varphi_n)$ is destabilizable $dist_e(\varphi_n)=1$. Therefore, $\phi_n$ and $\varphi_n$ are non-isotopic monodromy maps such that $\operatorname{Ob}(\mathop{\natural}_{n-1} (S^1\times D^1),\phi_n)$ and $\operatorname{Ob}(\mathop{\natural}_{n-1} (S^1\times D^1),\varphi_n)$ are diffeomorphic. Notably, their result showed the existence of a non-destabilizable planar open book with $n$ binding components, supporting the standard tight contact structure on $S^3$.
\end{remark}

\begin{figure}
    \begin{subfigure}{0.45\textwidth}
            \centering
            \includegraphics[width=0.55\textwidth]{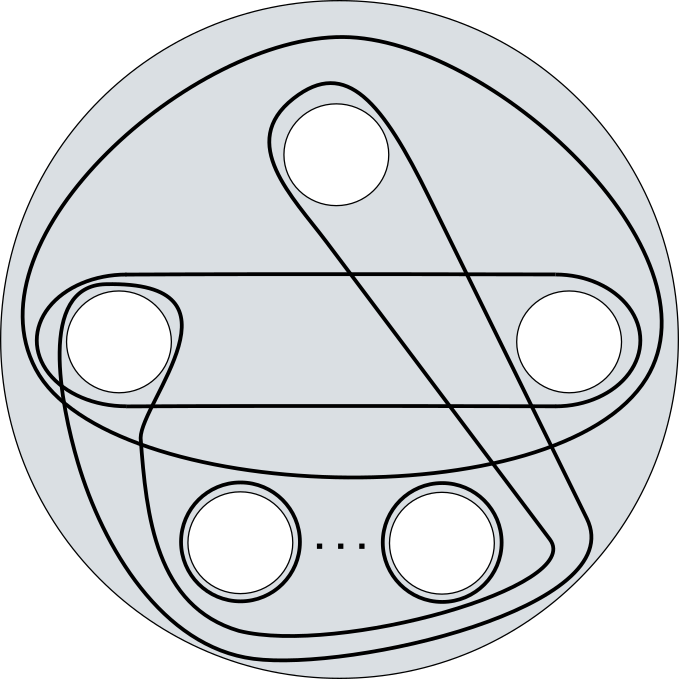}
    \end{subfigure}
        \hfill
    \begin{subfigure}{0.45\textwidth}
            \centering
            \includegraphics[width=0.55\textwidth]{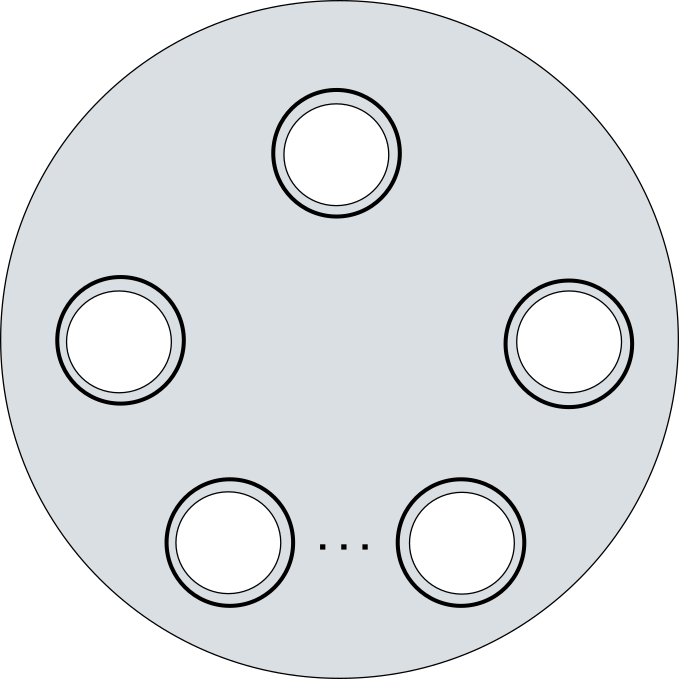}
        \end{subfigure}
    \caption{}
    \label{fig:planar_monodromy}
\end{figure}

\section{Application}
It was observed in~\cite{hsueh2023kirby} that for each $4$-dimensional manifold $X$ admitting an open book decomposition with trivial monodromy, there exists a monodromy $$\varphi\colon\mathop{\natural}_{\mu_1}(S^1\times D^{2})\mathop{\natural}_{\mu_2}(S^2\times D^1)\rightarrow\mathop{\natural}_{\mu_1}(S^1\times D^{2})\mathop{\natural}_{\mu_2}(S^2\times D^1)$$ on a $3$-dimensional genus $\mu_1$ handlebody with $\mu_2$ punctures given by the composition of twists along embedded spheres and tori such that $$(\mathop{\natural}_{\mu_1}(S^1\times D^{2})\mathop{\natural}_{\mu_2}(S^2\times D^1),\varphi)$$ is an open book decomposition on $X$. We develop a higher-dimensional analogue of this result. Furthermore, we show that two open books $(M,\operatorname{id}_{M})$ and $(N,\operatorname{id}_{N})$ can be stabilized to have a common page given that the Euler characteristic of the pages are equal -- modulo 2 if the pages are even-dimensional.

\subsection*{Proof of Theorem~\ref{thm generalization}}
    Suppose $\varphi=\operatorname{id}$ and let $\operatorname{h}\in\mathcal{H}(\nu_1,\nu_{2},\dots,\nu_{n-2})$ be a handle decomposition on $M$. Then the selection $\mathcal{A}=\{(j,k)|2\leq k\leq n-2, 1\leq j\leq \nu_k\}$ of all handles of indices greater than or equal to two is exchangeable and the exchanged page is given by        $$\calA{M}=\mathop{\natural}_{\nu_1}({S}^1\times D^{n-2} )\mathop{\natural}_{\nu_{2}}({S}^{n-2}\times D^{1})\mathop{\natural}\dots \mathop{\natural}_{\nu_{n-2}}({S}^{2}\times D^{n-3}).$$
    Recall that each boundary connected summand $S^{n-k}\times D^{k-1}$ was obtained in exchange for a \handle{k} of index $k\geq2$ and that the result of attaching $\nu_1$ \handles{1} to a \handle{0} $D^{n-1}$ is always diffeomorphic to $\mathop{\natural}_{\nu_1}(S^{1}\times D^{n-2})$~\cite[Example~4.1.4.(b)]{gs}. Let $\mu_1=\nu_1$ and $\mu_i=\nu_{n-i}$ for all $i\in[2,n-2]$ and rewrite $\calA{M}$ as $\mathop{\natural}_{i=1}^{n-2}\mathop{\natural}_{\mu_i}{(S^{i}\times D^{n-1-i})}$. By Corollary~\ref{cor partially simplifying the page}, there exists a monodromy map $\calA{\varphi}$ on $\calA{M}$ such that 
    $$(\mathop{\natural}_{i=1}^{n-2}\mathop{\natural}_{\mu_i}(S^{i}\times D^{n-1-i}),\calA{\varphi})$$
    is an open book decomposition on $X$. The diffeomorphism type of $\mathop{\natural}_{i=1}^{n-2}\mathop{\natural}_{\mu_i}(S^{i}\times D^{n-1-i})$ is well-defined without having to specify the boundary components, therefore $\calA{\varphi}$ only depends on $\mu=(\mu_1,\mu_2,\dots,\mu_{n-2})$. It follows that if $M$ admits a handle decomposition $\operatorname{h}\in\mathcal{H}(\mu_1,\mu_{n-2},\dots,\mu_{2})$, then $$\operatorname{Ob}(\mathop{\natural}_{i=1}^{n-2}\mathop{\natural}_{\mu_i}(S^{i}\times D^{n-1-i}),\sigma_\mu)\cong\operatorname{Ob}(M,\operatorname{id})$$ for some monodromy map $\sigma_\mu$ on $\mathop{\natural}_{i=1}^{n-2}\mathop{\natural}_{\mu_i}(S^{i}\times D^{n-1-i})$. \qed

\begin{lemma}\label{lemma euler char}
    If $M$ and $N$ have the same Euler characteristic, then they admit handle decompositions with the same number of handles of each index.
\end{lemma}
\begin{proof}
    Given a handle decomposition $\operatorname{h}\in\mathcal{H}(\mu_1,\dots,\mu_{n-2})$ on $(n-1)$-dimensional page $M$ and a natural number $t>\mu_j$, for some $1\leq j< n-2$, we obtain a new handle decomposition $\operatorname{h}'\in\mathcal{H}(\mu_1,\dots,\mu_{j-1},t,\mu_{j+1}+(t-\mu_j),\dots,\mu_{n-2})$ on $M$ by adding $(t-\mu_j)$ canceling pairs of $j$- and \handles{(j+1)} to $\operatorname{h}$.

    Choose handle decompositions $$\operatorname{h}_M\in\mathcal{H}(\mu_1,\dots,\mu_{n-2}) \text{ and }\operatorname{h}_N\in\mathcal{H}(\nu_1,\dots,\nu_{n-2})$$ on $M$ and $N$, respectively. Start with the smallest index $j>0$ with $\mu_j\neq\nu_j$. Without loss of generality assume $\mu_j>\nu_j$ and replace $\operatorname{h}_N$ with a handle decomposition on $N$ with $\mu_j$ $\handles{j}$ and $(\nu_{j+1}+\mu_j-\nu_j)$ \handles{(j+1)} as explained above. By repeating this process, we obtain handle decompositions $\operatorname{h}_M'$ and $\operatorname{h}_N'$ on $M$ and $N$, respectively, with an equal number of handles in each index less than $(n-2)$. Our hypothesis $\chi(M)=\chi(N)$ implies that $\operatorname{h}_M'$ and $\operatorname{h}_N'$ also have the same number of $\handles{($n-2$)}$.
\end{proof}

\subsection*{Proof of Corollary~\ref{cor common page}}
We want to show that $(M,\operatorname{id})$ and $(N,\operatorname{id})$ can be stabilized to open books with a common page of the form $\mathop{\natural}_{i=1}^{n-2}\mathop{\natural}_{\mu_i}(S^{i}\times D^{n-1-i})$ for some $(\mu_1,\dots,\mu_{n-2})$. The case where $M$ and $N$ are odd-dimensional follows from Theorem~\ref{thm generalization} and Lemma~\ref{lemma euler char}.

Suppose $M$ and $N$ are even-dimensional and have the same Euler characteristic modulo 2. By Theorem~\ref{thm generalization}, there exists $\mu=(\mu_1,\dots,\mu_{n-2})$ and $\nu=(\nu_1,\dots,\nu_{n-2})$ such that $(\mathop{\natural}_{i=1}^{n-2}\mathop{\natural}_{\mu_i}{(S^{i}\times D^{n-1-i})},\sigma_\mu)$ and $(\mathop{\natural}_{i=1}^{n-2}\mathop{\natural}_{\nu_i}{(S^{i}\times D^{n-1-i})},\sigma_\nu)$ are open book decompositions on $X$ and $Y$, respectively. It is elementary to check that $\chi(\mathop{\natural}_{i=1}^{n-2}\mathop{\natural}_{\mu_i}(S^{i}\times D^{n-1-i}))\equiv\chi(\mathop{\natural}_{i=1}^{n-2}\mathop{\natural}_{{\nu_i}}(S^{i}\times D^{n-1-i}))\mod2$. Observe that the Euler characteristic of the page $M$ and its $k$-th stabilized page $$M\mathop{\natural}({S}^{k-1}\times D^{n-k})\mathop{\natural}({S}^{n-k}\times D^{k-1})$$ are related by $$\chi({M})+(-1)^{(k-1)}\cdot 2=\chi(M\mathop{\natural}({S}^{k-1}\times D^{n-k})\mathop{\natural}({S}^{n-k}\times D^{k-1}))$$
because $k-1\equiv n-k$ modulo 2. Hence we can stabilize $\mathop{\natural}_{i=1}^{n-2}\mathop{\natural}_{\mu_i}(S^{i}\times D^{n-1-i})$ and $\mathop{\natural}_{i=1}^{n-2}\mathop{\natural}_{{\nu_i}}(S^{i}\times D^{n-1-i})$ sufficiently times such that the resulting pages have equal Euler characteristic. Then Lemma~\ref{lemma euler char} completes the proof of Corollary~\ref{cor common page}. \qed

\let\MRhref\undefined
\bibliographystyle{hamsalpha}  
\bibliography{Sources}  
\end{document}